\documentclass[11pt]{cpres}

\newcommand{\vol}{\operatorname{vol}}
\title{Guts, volume  and Skein Modules of 3-manifolds}

\begin{document}

\author{Brandon Bavier}

\author{Efstratia Kalfagianni}

\thanks{The research of both authors is partially supported by
 NSF Grants DMS-1708249 and  DMS-2004155}

\address[]{Department of Mathematics, Michigan State University, East
Lansing, MI, 48824, USA}
\email[]{kalfagia@msu.edu}

\address[]{Department of Mathematics, Michigan State University, East
Lansing, MI, 48824, USA}
\email[]{bavierbr@msu.edu}

\begin{abstract} We consider hyperbolic links that admit alternating projections on  surfaces in  compact, irreducible 3-manifolds.
We show that, under some mild hypotheses, the volume of the complement of such a link
is bounded below in terms of a Kauffman bracket function defined on link diagrams on the surface.

 In the case that the 3-manifold is a thickened surface, this Kauffman bracket function leads to a Jones-type  polynomial that is an isotopy  invariant of links. We show that coefficients of this polynomial provide 2-sided linear bounds on the volume of hyperbolic alternating links in the thickened surface. As a corollary of the proof of this result,
we deduce that the twist number of a reduced, twist reduced, alternating link projection with checkerboard disk regions, is an invariant of the link. 
\end{abstract}

\maketitle

\section{Introduction}

The goal of the paper is to show that, under mild hypotheses, the volume of
 a hyperbolic link in a compact, irreducible 3-manifold $M$ that admits an alternating projection on a closed  surface $F\subset M$,
is bounded below in terms of a Kauffman bracket function defined on link diagrams on $F$. For $M=F\times [-1, \ 1]$, this function leads to a Jones
polynomial link invariant and coefficients of it
 provide 2-sided linear bounds on the volume of hyperbolic alternating links. As a corollary, we obtain that the twist-number of a reduced, twist-reduced, alternating link projection with  checkerboard  disk regions, is an invariant of the link in $F\times [-1, \ 1]$. Our results generalize work of Dasbach and Lin \cite{VTJPAK} and Futer, Kalfagianni and Purcell \cite{DFVJP, Guts} who obtained similar results for families of links in $S^3$.

Let $M$ be an irreducible compact 3-manifold with or without boundary. A link $L$  admits  a projection on an orientable embedded surface $F$ in $M$, if
$L\subset F\times [-1, 1]\subset M$ and it is projected via the obvious projection $\pi:  F\times [-1, 1] \longrightarrow  F=F\times\{0\}.$
Given a connected surface $F$ in $M,$ we define
 a Kauffman bracket function and from this we construct a   polynomial
\[ J_0(\pi(L))=\langle \pi(L) \rangle_0= a_mt^m+ a_{m-1}t^{m-1} + \ldots + b_{n+1}t^{n+1} + b_n t^{n},\]
in $\bZ[t^{\frac{1}{4}},t^{-\frac{1}{4}}]$.
See Section 2 for the precise definition.

The polynomial $J_0(\pi(L)),$  depends on  the topology of $F$, the projection $\pi(L)$ and apriori on the topology of the complement  $M\setminus N(F)$.
In the special case that $M=F\times [-1, \ 1],$ it is 
an isotopy invariant of $L$ in $M$, but
we don't expect that it is an isotopy invariant  of $L$ in a general $M$.
Nevertheless, as our results below show, if $\pi(L)$ is alternating on $F$, and  under mild additional hypotheses,  $J_0(\pi(L))$  encodes intrinsic geometric information of the link complement $M\setminus L$.

Our first result is the following theorem, where the terms reduced and twist-reduced are defined in Section 3.

\begin{thm} 
\label{thm:volume} 
Let $M$ be an irreducible, compact 3-manifold with empty or incompressible boundary. 
Let $F\subset M$ be an incompressible, closed orientable surface such that $M\setminus N(F)$ is atoroidal
and $\bndry$-anannular. Suppose that  a link $L$ admits a  reduced alternating projection $\pi(L)\subset F$  that is checkerboard colorable, twist-reduced and with all the regions of $F\setminus \pi(L)$
disks.
Then $L$ is hyperbolic
and we have

\[\vol(M\setminus L) \ \geq v_8\  
{\rm{max}}\big\{ |a_{m-1}|-|a_{m}|, \   |b_{n+1}|- |b_{n}|\big\}-\frac{1}{2}\chi(\bndry M),\]
where $a_{m-1},  a_m, b_{n+1}, b_n$ are the two first and two last coefficients of the polynomial  $J_0(\pi(L)),$ and  $v_8=3.66386\cdots$ is the volume of a regular ideal octahedron.

\end{thm}

Given an alternating link projection   $\pi(L)$  as in the statement of Theorem ~\ref{thm:volume}, let $S_A$, $S_B$ denote the two checkerboard surfaces corresponding to $\pi(L)$. Also let $M_A $ and $M_B$ denote the manifolds obtained by cutting $M\setminus L$ along $S_A$ and $S_B$, respectively.   By  ~\cite{WGA},  $S_A$, $S_B$ are essential in $M\setminus L$, and by
 Jaco-Shalen-Johannson 
theory $M_A$ and  $M_B$ contain hyperbolic sub-manifolds called the {\em{guts}} of $S_A$ and $S_B$, respectively.
We show that the Euler characteristics of these guts and the twist number $t_F(\pi(L))$ can be calculated from  $J_0(\pi(L))$.

\begin{thm}
\label{thm:GutsAndJones}
Let $M$, $F$ and $\pi(L)$ be as in the statement of Theorem ~\ref{thm:volume} and let  $t_F(\pi(L))$
denote the twist number of $\pi(L)$. Also let $\guts(M_A)$ and $\guts(M_B)$ denote the guts of $S_A$ and $S_B$, respectively.
We have the following.
\begin{enumerate}
\item $\chi(\guts(M_A)) = |a_m|-|a_{m-1}| + \frac{1}{2}\chi(\bndry M)$,
\vskip 0.05in
\item $\chi(\guts(M_B)) =|b_n|-|b_{n+1}| + \frac{1}{2}\chi(\bndry M)$,  
\vskip 0.05in
\item $t_F(\pi(L))=|a_{m-1}| + |b_{n+1}| - |a_m| - |b_m|+\chi(F).$ 
\end{enumerate}
\end{thm}

Theorem  ~\ref{thm:volume} follows by combining Theorem ~\ref{thm:GutsAndJones} with a result of Agol, Storm and Thurston \cite{AST} asserting that
the negative Euler characteristic of the guts of an essential surface in a hyperbolic 3-manifold $M$ bounds linearly the volume of $M$ from below.

For $M=S^3$ and $F=S^2,$  the polynomial $J_0$ is  the classical Jones polynomial.
In the case that   $M=F\times I$, one can use the structure of the Kauffman skein module of $M$, to see that   $J_0(\pi(L))$ is also an isotopy invariant of the link $L$. 
In fact, for every link  in $M$ one obtains a finite collection of Jones-type polynomial invariants that have been used to settle open questions about the topology of alternating links in thickened surfaces \cite{ThickTaitConj, AdamsFXI}.
 As a corollary of Theorem  \ref{thm:GutsAndJones} we obtain the following.

\begin{restatable}{cor}{invariant}
Let $L$ be a link in $F\times[-1,1]$, that admits a
 checkerboard colorable, reduced
alternating projection $\pi(L) \subset F$ that is twist-reduced has all its regions
disks.
Then, any two such projections of $L$ have the same twist number.
That is, $t_F(\pi(L))$ is an isotopy invariant  of $L$.
\end{restatable}

Note that Theorem \ref{thm:GutsAndJones} also implies that the quantities
 $\chi(\guts(M_B))$ and  $\chi(\guts(M_B))$ are invariants of $L$ in $M=F\times[-1,1]$.
 
For reduced, twist-reduced alternating diagrams on a 2-sphere in $S^3$,  invariance of the twist number
is a consequence of  the Tait flyping conjecture  \cite{ClassAltLinks}. The corresponding conjecture  for links in
thickened surfaces is currently open.
A second proof of the twist number invariance  for alternating links in $S^3,$
follows from the work of Dasbach and Lin \cite{VTJPAK, DLadeq} that relates this twist number to the Jones polynomial. Our approach generalizes their approach.

Several families of hyperbolic links in $S^3$, including alternating ones, are known to satisfy a ``coarse volume conjecture": coefficients of the Jones and colored Jones polynomial provide two-sided linear  bounds of the volume of the link complement ~\cite{VTJPAK, DFVJP, CAFM, SLCS, Lee17, Giam, Guts}. The next theorem provides a similar result for alternating  links  in thickened surfaces and there is a similar result
for links with alternating projections on  Heegaard tori in Lens spaces (see \ref{Cor:TorAltUpperBound} ). 

\begin{thm}
\label{Thm:FxIUpperBound}
Suppose that $\pi(L)$ is a reduced 
alternating projection on $F=F\times\{0\}$ in $M=F\times[-1,1],$ that is  twist-reduced, checkerboard colorable, and with all the regions of $F\setminus \pi(L)$
disks.
Then the interior of $M\setminus L$ admits a hyperbolic structure. If $F=T^2$, then we have

\[  {\frac{v_8}{2}}\cdot \beta_L \leq \vol(M\setminus L) < 10\,v_4\cdot \beta_L, \]
and if  $F$ has genus at least two, 
\[  \frac{v_8}{2}\cdot \left( \beta_L -2\chi(F)\right) \leq \vol(M\setminus L) < 6\,v_8\cdot \left( \beta_L+\chi(F)\right) ,\]
where $\beta_L:= |a_{m-1}| + |b_{n+1}| - |a_{m}| - |b_{n}|$, is obtained from the Jones polynomial invariant $J_0$ of $L$, and
$v_4 = 1.01494\dots$ is the volume of a regular ideal tetrahedron.
\end{thm}

A key idea in  the proof of Theorem  \ref{thm:GutsAndJones},   is to relate the coefficients of $J_0(\pi(L))$ to the topology of the checkerboard graphs of
 any  projection $\pi(L)\subset F$. This idea is reminiscent of techniques that were used to study the Jones polynomial of adequate and alternating links in $S^3$\cite{Lickbook, LT, 5AMP, Guts}.
 Under  a graph theoretic condition, that we call geometric adequacy,  we show that the first and last two coefficients of  $J_0(\pi(L)$,  can be calculated from the checkerboard graphs.
 The checkerboard graphs of reduced, alternating projections turn out to satisfy these graph theoretical conditions. On the other hand, these graphs form spines of the checkerboard surfaces.
 We use the work of Howie and Purcell \cite{WGA} in  a crucial way to we show that the graph-theoretic combinatorics that determine the coefficients $a_{m-1}$ and $b_{n+1}$ are exactly the ones dictating the calculation
 of the Euler characteristics of the guts of the checkerboard surfaces.
 
 There exist  open conjectures  predicting that  the volume of hyperbolic 3-manifolds is determined by certain asymptotics of quantum invariants  ~\cite{ kashaev:vol-conj, book, Chen-Yang, MR3874003}. For links in $S^3$ these invariants include the Jones polynomial and its generalizations. The relations of skein theoretic invariants and volume via guts of surfaces
  established in  \cite{DFVJP, Guts} and in Theorems \ref{Thm:FxIUpperBound} and \ref{thm:GutsAndJones}, are robust and  seem independent of these  conjectures.
 
 The paper is organized as follows: in Section 2, we consider projections of links on a surface and we define the polynomial $J_0(\pi(L)).$ Then, we explain how known results on the structure of  Kauffman skein modules imply that   $J_0(\pi(L))$
 is an invariant of isotopy for links in $F\times [-1, \ 1]$.  Finally, we restrict ourselves to projections $\pi(L)\subset F$ that have disk regions. For such projections we define   the notion
 to geometric adequacy,  under which  we obtain formulae for the coefficients $|a_{m-1}|,  |b_{n+1}|, |a_m|, |b_n|$ (Theorem \ref{prop:JonesPoly}).
 We also compare the notion of geometric adequacy with other  notions of adequacy that have recently appeared in the literature \cite{ThickTaitConj}.
 
 In Section 3 we consider alternating projections  $\pi(L)\subset F\subset M$ that are checkerboard colorable and have disk regions. We define a notion of diagram reducibility that generalizes the corresponding notion for link diagrams
 on a 2-sphere in $S^3$, and interplays nicely with the complexity of ``edge representativity" considered in \cite{WGA} and with geometric adequacy (Proposition \ref{prop:ReducedEdgeRep} and Lemma \ref{conditions}).
 This interplay allows us to relate our work in Section 2 to work of Howie and Purcell on weakly generalized alternating links. The main result in this section is Theorem \ref{thm:Gutsgeneral} 
 that is a generalized version of Theorem \ref{thm:GutsAndJones}. The more general version replaces the hypothesis that $F$ is incompressible in $M$ with a hypothesis of ``high representativity". We also prove Corollary 1.3. In Section 4 we prove Theorem \ref{thm:volume}  and we derive Theorem \ref{Thm:FxIUpperBound}.  

\smallskip
{\bf Acknowledgement.} The authors thank Renaud Detcherry  for helpful conversations about skein modules of 3-manifolds.

\medskip

\section{Skein polynomials and geometric adequacy}

\subsection{Bracket polynomials for link diagrams on surfaces}  Let $M$ be a 3-manifold with or without boundary and let $F$ be an orientable surface 
in $M$. Given a link $L\subset F\times [-1, \ 1]\subset  M$  we can consider its image under the projection $\pi:  F\times [-1, 1] \longrightarrow  F=F\times\{0\}$.
Throughout the paper we will refer to
 $\pi(L)\subset F$ as a link projection or a link diagram.  
 
\begin{figure}[H]
	
	\centering
		\def\svgwidth{.7\columnwidth}
		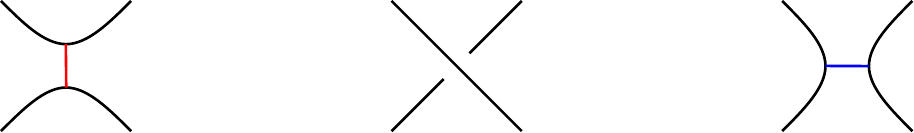

		\caption{ The  $A$ resolution(left) and the $B$ resolution (right) of a crossing.}
		\label{fig:resolutions}
		
\end{figure}

 Given a crossing of a link diagram  $D=\pi(L)$, we define the $A$ and $B$ resolutions of the crossing  as indicated in Figure ~\ref{fig:resolutions}. 

Let  ${\mathcal D}(F)$ denote the set
of all (unoriented) link diagrams on $F$,  taken up isotopy on $F.$ Also  let  $X_F$ denote  the set of all collections of disjoint simple closed curves (a.k.a multi-curves) on $F$. We define a Kauffman  bracket function 
 
 $${\langle \ \  \rangle}: {\mathcal D}(F)\longrightarrow  \bZ[A, A^{-1}, (A^2-A^{-2})^{-1}]X_F,$$
 by the following skein relations.
\begin{enumerate}
	\item $\langle \KPposcrossing \rangle = A \langle \KPvertical \rangle + A^{-1} \langle \KPhorizontal \rangle$
	\item $\langle L \sqcup \KPcircle \rangle = (-A^2 - A^{-2}) \langle L \rangle$
	\item $\langle  \KPcircle \rangle=1$
\end{enumerate}
As we are working on a surface that may have essential curves, we also require that the unknots in the above relations bound disks on $F$.

 To describe ${\langle D \rangle}$ as a function in $ \bZ[A,A^{-1}, (A^2-A^{-2})^{-1}]X_F$ in more detail we need some preparation.
 
\begin{defn}\label{defn:states}A  \textit{Kauffman state} for a link diagram diagram $D\subset F$ is an assignment of the  $A$ or $B$ resolution for each crossing of $D$. The result of applying any state to $D$ is a collection of disjoint simple closed curves  on $F$ called  \textit{state circles}.

Given a state $s$ we will use $|s|=|s(D)|$  to denote the number of state circles resulting from $D$ by applying $s$, and  we let  $a(s)$ be the number of $A$ resolutions in $s$, and $b(s)$ the number of $B$ resolutions.
Also, we will use $s_t$ and $s_{nt}$ to denote the set of contractible and non-contractible state circles resulting from $s$, and we will use $|s_t|$ and $|s_{nt}|$ to denote the cardinalities of these sets.

Finally, we will use  $s_A$ to denote the state where all the resolutions are $A$, and $s_B$ to denote the state where all the resolutions are $B$.
\end{defn}

\vskip 0.02in

Given a link projection $D=\pi(L)$ in ${\mathcal D}(F)$ we define

\begin{equation}
\label{eqn:B0}
{\langle D \rangle}_0 = \sum_{\{s \;  | \; s_{nt}=\emptyset\}} A^{a(s)-b(s)} (-A^2 - A^{-2})^{|s_t|-1},
\end{equation}
that is, we sum over all states that when applied to $D$ produce only contractible state circles.
Similarly given a collection  $X$ of simple closed disjoint curves on $F$, none of which is contractible, we define

\begin{equation}
\label{eqn:BX}
{\langle D \rangle}_X= \sum_{\mathrm{\{s \;  | \; s_{nt}=X\}}} A^{a(s)-b(s)} (-A^2 - A^{-2})^{|s_t|-1} ,
\end{equation}
where $|s_t|$ is the number of contractible  curves in $s=s(D).$ Clearly, for a given $D$, the value of ${\langle D \rangle}_X$ will be non-zero for only finitely many collections $X$.

Using the defining skein relations,  for any $D\in {\mathcal D}(F)$ we can write $\bZ[A,A^{-1}]$

\begin{equation}
\langle D \rangle =\langle D \rangle_0+   \sum_{X\in X_F}\   {\langle D \rangle}_X  X.
\label{bracketD}
\end{equation}

Note the by definition, ${\langle D \rangle}_0$ is always a polynomial in $\bZ[A,A^{-1}]$ and
${\langle D \rangle}_X$ are not   in $\bZ[A,A^{-1}]$ if there exist states for which $s_t=0$,
in which case we get the factor $(-A^2 - A^{-2})^{-1}$. Thus, in particular,  $(-A^2 - A^{-2}){\langle D \rangle}_X$ and always lies in  $\bZ[A,A^{-1}]$.

The argument used for $S^2$ in the 3-sphere works to show that
${\langle D \rangle}$
is invariant under Reidemeister moves II and III  on $F$ and that it changes by a power of $A$ under  Reidemeister move I. 
Then, ${\langle D \rangle}={\langle D \rangle}_0$ is an invariant of framed links and a normalization of it gives the Jones polynomial  (see e.g.  ~\cite{Lickbook}).
We don't expect that ${\langle \pi(L) \rangle}$ an isotopy invariant of the framed link $L$ in a general 3-manifold $M$.

A way to generalize the Jones polynomial for links in arbitrary 3-manifolds is  to consider skein modules: given a 3-manifold $M$, let ${\mathcal L}(M)$
denote the set of isotopy classes of framed links in $M$.
The \emph{Kauffman skein module} of a 3-manifold $M$, denoted by  $\Skein(M)$, is the quotient of the free $\bZ[A,A^{-1}]$ module generated by  
${\mathcal L}(M)$,
by the submodule generated by all relations of the form
\begin{itemize}
	\item $\KPposcrossing -A  \KPvertical  -A^{-1}  \KPhorizontal $
	\item $ L \sqcup \KPcircle - (-A^2 - A^{-2})  L $
	\item $\KPcircle - 1$
\end{itemize}
Here the crossing modifications take place in a small 3-ball in $M$ that intersects the link to be modified at a single crossing and the notation $ \KPcircle$ is used to define the
isotopy class of the knot that bounds a smooth disk in $M$.
Given a link $L\subset M,$ let ${\bar L}$ denote the class of $L$ in $M$. Now the image of ${\bar L}$ 
under the map 
$${\mathcal L}(M) \longrightarrow \Skein(M), $$
is an isotopy invariant of the framed link $L$.

Let us now discuss the special case where $M=F\times [-1, \ 1]$, a thickened surface.
It is known that  $\Skein(F\times [-1, 1])$ is free over $\bZ[A,A^{-1}]$ with basis $X_F\cup \{\emptyset\}$, where $\emptyset$ is the empty knot ~\cite{Pr99}.
Now any framed link in $F\times [-1, 1]$ can be projected on $F=F\times \{0\}$, and any two link diagrams of $F$ represent the same element in
${\mathcal L}(M)$ if and only if they are related by Reidemeister moves II and III on $F$.
Thus  $D=\pi(L)$, the expression in equation (\ref{eqn:BX}) can immediately be viewed as the image of $L$ in $\Skein(F\times [-1, 1])$. 
It follows that  ${\langle D \rangle}_0 , {\langle D \rangle}_X$  are invariants of $L$ in $F\times [-1, 1]$. 
These invariants were recently considered by Boden, Karimi and Sikora  ~\cite{ThickTaitConj}
and were used to prove versions of two of the Tait conjectures for alternating links $F\times [-1, 1]$.

We will be concerned with $\langle D \rangle_0$, the sum of elements that appear when, after using the skein relations, we are left with just the empty knot. 
As pointed out in ~\cite{ThickTaitConj}, one  can make $\langle D \rangle_0$ an isotopy invariant of oriented non-framed links by 
considering
$(-1)^{w(D)}A^{-3w(D)} \langle D \rangle_0,$ where $w(D)$ is the writhe number of $D$.
Setting
$$J_0(t)=\left((-1)^{w(D)}A^{-3w(D)} \langle D \rangle_0\right)|_{t=A^{4}},$$
we obtain a isotopy invariant of $L$. Note, our convention differs from ~\cite{ThickTaitConj}, as they set
$t=A^{-4}$.  Let

$$J_0= a_mt^m+ a_{m-1}t^{m-1} + \ldots + b_{n+1}t^{n+1} + b_n t^{n},$$
where $m$ and $n$ denote the highest and lowest degrees of $J_0$, respectively. 

\begin{prop}\label{isinvariant} For any link $L\subset M=F\times [-1, \ 1]$, the polynomial $J_0=J_0(L)$ is an isotopy invariant of $L$.
\end{prop}

\begin{rem}{\rm  We are interested in the absolute values of the coefficients of $J_0$ which are the same as these of
${\langle D \rangle}_0 $, since as mentioned earlier remain unchanged under Reidemeister move I on $F$. By slightly abusing our setting, when talking about these coefficients, we will feel free to use $J_0(D)$ or  ${\langle D \rangle}_0$ interchangeably. }
\end{rem}

In general let us  start with a 3-manifold $M$ and a connected, closed, orientable surface $F$ embedded in $M$ and a projection
of $\pi(L) \subset F$ of a link $L\subset F\times [-1, 1]\subset M$. We can define ${\langle \pi(L) \rangle}_0 $ and  $J_0(D)$ as above, but general it is hard to decide when they descend to isotopy link invariants in $M$: firstly, understanding the structure of $\Skein(M)$ is known to be a very hard problem. There is no algorithm for computing
$\Skein(M)$  in general, and these modules have only been explicitly computed  for some simple families of 3-manifolds. See \cite{RW} and references therein.
For our purposes here, we  will consider $F\times [-1, 1]$ embedded in a 3-manifold $M$ that is closed or it has incompressible boundary.  The inclusion induces a map
\begin{equation}
 \Skein(F\times [-1, 1])\longrightarrow \Skein(M),
 \label{eqn:inclusion}
 \end{equation}
and the image of $L$ in $\Skein(F\times [-1, 1])$ is easy to calculate as we said above.
However, in general very little is known about the structure of the map of Equation (\ref{eqn:inclusion}).
 For instance, for closed $M$, if one works over the field $\bQ(A)$, then the skein module of $M$ is finitely generated while the skein module of $F\times [-1, 1]$ is infinitely generated \cite{GJS}.
Thus, in this case, the map in equation (\ref{eqn:inclusion}) has a substantial kernel and it is expected that this is the case over
$\bZ[A,A^{-1}]$ as well.
Given a  link $L\subset F\times [-1, \ 1]\subset M$, with $D=\pi(L)$ on $F$,
in general we don't expect that
 $J_0(\pi(L))$ is  an invariant of isotopy of $L$ in $M$. 
In the remaining of the paper, unless working with $M=F\times [-1, \ 1]$, we will consider $J_0$  as a function on the set of link projections $\pi(L)$ on $F.$
In this setting,
and it is rather striking that, as
Theorems ~\ref{thm:volume} and ~\ref{thm:volumegeneral} show, 
$J_0(\pi(L))$  captures  intrinsic geometric information of the complement $M\setminus L$.
\smallskip

\subsection{Geometrically adequate links}
Recall that $M$ is a 3-manifold and $F\subset M$ is an embedded orientable surface and let $D=\pi(L)\subset F$ be a link projection.
Given a Kauffman state $s$ on $D$ we define the \textit{ state graph} $G_s=G_s(D)$ as follows: The vertices of $G_s$ correspond to state circles of $s$,
and the edges correspond to crossings
of $\pi(L)$. Each edge connects the sub-arcs of the state circles that remaining from the splitting of that crossing in $s$.
We will use $G_A$ to denote the graph corresponding to $s_A$, and $G_B$ to denote the graph corresponding to $s_B$.
From now on we will  restrict ourselves to projections $D$ where  all the state circles in $s_A$ and in $s_B$ are contractible  on $F$.
See Figure \ref{alldisks} for an example of a link diagram $D$ with this property, where we also show the graphs $G_A$, $G_B$.

\begin{defn} 
We say that the diagram $D=\pi(L)\subset F$ is \textit{geometrically} $A$\textit{-adequate} if $G_A(D)$ has no 1-edge loops and all circles of $s_A$ are contractible. Likewise, we say $D$ is \textit{geometrically} $B$\textit{-adequate} if $G_B(D)$ has no 1-edge loops and all circles of $s_B$ are contractible.

If $D$ is both geometrically $A$-adequate and $B$-adequate, we say it is \textit{geometrically adequate}.
\end{defn}

\begin{defn}With the notation as above, suppose that the diagram $D=\pi(L)\subset F$ is geometrically adequate.
Define the \textit{reduced graph of} $G_A$, $G'_A$, to be the graph where, if two edges $e_1$ and $e_2$ are adjacent to the same pair of vertices, we remove one of them if $e_1\cup e_2$ bounds a disk on $F$. Let $e'_A$  the number of edges of $G'_A$.

Similarly define the reduced graph $G'_B$ and denote the number of edges by  $e'_B$.
\end{defn}

In $S^3$ a link diagram $D=D(L)$ on a projection 2-sphere is called is called $A$-adequate, if $G_A(D)$ has no 1-edge loops and $B$-adequate
$G_B(D)$ has no 1-edge loops.  For such diagrams, Futer, Kalfagianni and Purcell  ~\cite{Guts} have established relations between coefficients of the Kauffman bracket of $D$ and 
geometric properties and invariants of the link complement $S^3\setminus L$.  In particular, they show that when $D$ represents a hyperbolic link, coefficients of the Kauffman bracket of $D$
provide linear bounds for the volume of the complement of the link. In this paper, we will generalize these geometric relations for links that admit alternating projections on surfaces in 3-manifolds.
A key step for this generalization is Theorem ~\ref{prop:JonesPoly} below that also holds for projections on a sphere in $S^3$.

\begin{thm} 
\label{prop:JonesPoly}
Suppose that  $D=\pi(L)\subset F\subset M$  and let $a_m, a_{m-1}, b_{n+1}, b_n$ denote the two first and two last coefficients of $J_0(D)$.
\begin{enumerate}

\item If $D$ is geometrically $A$-adequate then, we have $|a_m|  = 1$ and $|a_{m-1}| = e'_A - |s_A| + 1$.

\item If $D$ is geometrically $B$-adequate then, we have $ |b_n| = 1$ and $|b_{n+1}| = e'_B - |s_B| + 1$.
\end{enumerate}
\end{thm}

Theorem  ~\ref{prop:JonesPoly} and its proof should be compared with ~\cite[Proposition 2.1]{VTJPAK} and the usual calculation of the degrees of the Jones polynomial for ordinary alternating links in ~\cite{Lickbook}. 
We will split the proof into two lemmas. The first one concerns  the determination of $|a_m|$ and $|b_n|$.

\begin{lem}
\label{lem:firstCoef} 
Suppose that  $D=\pi(L)\subset F\subset M$  is a link diagram and   let $a_m,  b_n$ denote the first and last coefficients of $J_0(D)$.
\begin{enumerate}
\item If  $D$ is geometrically $A$-adequate then, we have $|a_m|  = 1$.
\item If  $D$ is geometrically $B$-adequate then,  we have $ |b_n| = 1$.
\end{enumerate}
\end{lem}

\begin{proof} 
Let $c=c(D),$ denote the number of crossings of $D$ and consider the all-$A$   state, $s_A$.  Then,  $a(s_A) = c$ and $b(s_A)=0$.  By the  definition of $\langle D\rangle_0$ (see
 Equation ~(\ref{eqn:B0}), the contribution of $s_A$ is
\[A^c(-A^2-A^{-2})^{|s_A|-1}\]
 The highest degree here, then, is $c+2|s_A|-2$, and the coefficient belonging to it is $(-1)^{|s_A|-1} = \pm 1$. 

Now we will show that all the  other states have degrees less that $c+2|s_A|-2$. We can view any state as being obtained from $s_A$ by a finite series of changing an $A$ resolution to a $B$ resolution. We can write this series out as $s_A \to s_1 \to s_2 \to \cdots \to s_k$. We will show that $s_{i+1}$ has degree at most $s_i$. First, note that $a(s_i) = a(s_{i-1}) - 1$ and $b(s_i) = b(s_{i-1}) + 1$. Next, by changing a single resolution, we are doing one of the following:

\begin{itemize}
	\item We merge two contractible circles to a contractible one (so $|s_i| = |s_{i-1}| - 1$).
	\item We split a contractible circle into two contractible ones (so $|s_i| = |s_{i-1}| + 1$).
	\item We split a contractible circle into two non-contractible ones (so $|s_i| = |s_{i-1}| - 1$).
	\item We merge two non-contractible circles into a contractible one (so $|s_i| = |s_{i-1}| + 1$).
	\item We merge a contractible and non-contractible circle into a non-contractible one (so $|s_i| = |s_{i-1}| -1$).
	\item We re-arrange  a non-contractible circle of  $s_{i-1}$ to a non-contractible one of  $s_i$ (so $|s_i| = |s_{i-1}|$).
\end{itemize}

In particular, each resolution change will either increase the number of state circles by one, leave it the same, or decrease it by one. As a result, the highest degree
in the contribution of $s_i$ to. $\langle D\rangle_0$ will be less or equal to this of $s_{i-1}$.

As the highest degree contribution  to $\langle D\rangle_0$ coming from $s_A$ is $c + 2|s_A|$, while the highest degree contribution coming from $s_k\neq s_A$ is $c - 2k + 2|s_k|-2$,
in order for $s_k$ to contribute to the highest degree of  $\langle D\rangle_0$, we would need to have
\begin{align*}
         c + 2|s_A| - 2 &\le c - 2k + 2|s_k| - 2 \\
\implies          |s_A| &\le |s_k| - k.
\end{align*}
This  would mean that each state change $s_{i-1} \to s_i$ must increase the number of state circles by by one, limiting what sort of changes we can make from the five possibilities discussed above.
However, since all the state circles in $s_A$ are contractible and $G_A$ has no 1-edge loops
the first change $s_A \to s_1$, will merge two contractible circles to one, so $|s_1| = |s_A| - 1$ and the highest degree of $s_1$ is strictly less that this of $s_A$. Since this degree cannot increase during the change from  
$s_A$ to $s_k$,
the contribution of $s_k$ has degree less than the contribution of $s_A$. So it does not contribute to the highest degree of $\langle D\rangle_0$, and we are done with part (1).

To see part (2), let $D^{*}\subset F\subset M$ denote the link diagram obtained from $D$ by switching all the crossings of $D$ simultaneously.
By the definition we can see that  $\langle D^{*}\rangle_0$  is obtained from $\langle D\rangle_0$  by changing $A$ to $A^{-1}$. Thus, we have $|b_n(D)| =|a_m(D^{*})|$ and the conclusion follows from part (1).
\end{proof}

Now we turn to the second lemma, that treats the second and penultimate coefficients of $J_0(D)$.

\begin{lem}
\label{lem:secondCoeff}
Suppose that  $D=\pi(L)\subset F\subset M$  is a link diagram  and let $a_{m-1}, b_{n+1}$ denote the second and penultimate  coefficients of $J_0(D)$.
\begin{enumerate}

\item If $D$ is geometrically $A$-adequate then, we have $|a_m|  = 1$ and $|a_{m-1}| = e'_A - |s_A| + 1$.

\item If $D$ is geometrically $B$-adequate then, we have $ |b_n| = 1$ and $|b_{n+1}| = e'_B - |s_B| + 1$.
\end{enumerate}
\end{lem}

\begin{proof}
Let $c=c(D)$ denote the number of crossings of $D$. We know that the highest degree of $\langle D\rangle_0$ is $c + 2|s_A|$. Then the second highest degree has exponent $c + 2|s_A| - 4$. A contribution to this degree can come from either the second highest degree of $s_A$, or from the highest degree of a state $s_k$ in which $k\neq 1$ crossings of $D$ are assigned the $B$ resolution.

First, we will deal with $s_A$. Recall that the part of $D_0$ coming from $s_A$ is
\begin{align*}
A^{c}(-A^2-A^{-2})^{|s_A|-1} &= (-1)^{|s_A|-1} A^c (A^2 + A^{-2})^{|s_A|-1}\\
	&= (-1)^{|s_A|}A^{c} \sum_{i=0}^{|s_A|} \binom{|s_A|-1}{i} A^{2|s_A|-2i-2} A^{-2i}\\
	&= (-1)^{|s_A|-1}\sum_{i=0}^{|s_A|-1} \binom{|s_A|-1}{i} A^{c+2|s_A|-4i-2}.
\end{align*}
In particular, the second highest degree is $c+2|s_A|-6$ (when $i=1$), and the coefficient is $(-1)^{|s_A|-1}(|s_A|-1)$.
Next, we deal with states  $s_k\neq s_A$. As in the proof of Lemma ~\ref{lem:firstCoef}, we can write $s_k$ as a finite series of resolution changes, $s_A\to s_1 \to \cdots \to s_k$. We claim that $s_k$ can only contribute to the degree 
$c+2|s_A|-6$ if all resolution changes happen to parallel edges in $G_A$ (i.e edges that are adjacent to the same pair of vertices and any two of which encircle a disc on $F$). 

Recall that the highest degree $s_k$ contributes is $c - 2k + 2|s_k|$. For this to contribute to the second highest degree of $\langle D\rangle_0$, we must have

\begin{align*}
         c + 2|s_A| - 6 &= c - 2k + 2|s_k|-2\\
\implies      |s_A|     &= |s_k| - k + 2.
\end{align*}

This means that, in our series from $s_A$ to $s_k$, we must either increase the number of state circles in all but exactly two resolution changes, where we don't change the number at all, or we must decrease the number of state circles exactly once, and the rest of resolution changes must increase it. As in Lemma ~\ref{lem:firstCoef}, we know that $s_A \to s_1$ must decrease the number of state circles. As such, for $s_k$ to contribute to the second highest degree the remaining resolution changes must increase the number of state circles.

As $s_A$ has no non-contractible circles, any following state will only have them if we introduce them from a resolution change. Note, that we need to change at least two resolutions from $s_A$ to create non-contractible  state circles.
However, turning these circles into contractible ones again will require merging and this step will decreasing the number of state circles. Thus such a sequence cannot contribute
to $a_{m-1}$. Thus, to increase the number of state circles, we must, after $s_A\to s_1$, always split a single contractible circle into two contractible circles. This can only happen in the change $s_{i-1}\to s_i$, however, if the there is a 1-edge loop in the state graph $G_{s_{i-1}}$. Such 1-edge loops are created when we merge two state circles together. In our series, this can only happen in the first state change, $s_A\to s_1$, and so all other changes must be adjacent to the same two state circles.
There are now two cases to consider. Let $e_1$ be the edge  of $G_A$ affected, during the change $s_A\to s_1$. Then either
\begin{itemize}
\item  $s_{i-1}\to s_i$ affects an edge $e_i$, with $e_1\cup e_i$ bounding a disk on $F$, or 
\item  $s_{i-1}\to s_i$ affects an edge $e_i$, with $e_1\cup e_i$ not bounding a disk on $F$.
\end{itemize}
The two cases are illustrated in Figure \ref{cases}. If we are in the second case, we will create two non-contractible circles, each parallel to the curve $e_i \cup e_2$. If we are in the first case, we create two contractible circles, and so are fine. As such, $s_k$ only contributes to the second highest degree if all resolution changes happen to parallel edges that bound a disk on $F$. We can view this as a single edge of $G_A'$.

\begin{figure}
\centering
\def\svgwidth{0.8\columnwidth}
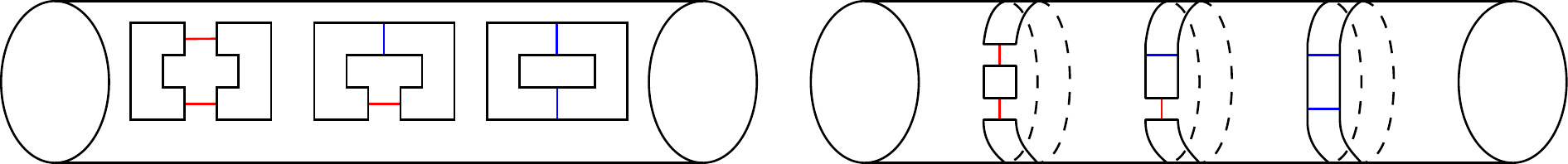
\caption{{The two cases of changing  two edges of $G_A$  adjacent to the same pair of vertices to edges of $G_B$.}}
\label{cases}
\end{figure}

Each family of parallel edges has several states associated to it: if the family has $k$ parallel edges (and thus $k$ crossings), and we want a state to have $1\le j\le k$ differences from $s_A$, then there are $\binom{k}{j}$ such states. While the highest degree remains the same, other values do change. If a state $s$ has $j$ changes in resolutions from $s_A$, then $a(s) = c-j$, $b(s)=j$, and $|s| = |s_A|-2+j$. Then the highest degree and coefficient such a state contributes is
\begin{align*}
(-1)^{|s|}A^{a(s)-b(s)}A^{2|s|-2} &= (-1)^{|s_A|+j-3}A^{c-2j+2|s|-2}\\
&= (-1)^{|s_A|-3}(-1)^j A^{c+2|s_A|-6}.
\end{align*}
Summing over all possible states for this family, then, the family contributes the coefficient
\[(-1)^{|s_A|-3}\sum_{j=1}^k \binom{k}{j} (-1)^{j}.\]
Using the binomial theorem,
\begin{align*}
0=(1+(-1))^k &= \sum_{j=0}^k \binom{k}{j} (-1)^j\\
&= 1 + \sum_{j=1}^k \binom{k}{j} (-1)^j\\
&= 1 + (-1)
\end{align*}
so the coefficient contributed by a single family is just $(-1)^{|s_A|-1}$.

Now, adding the coefficient we get from $s_A$ to  all the coefficients we get from edges of $G_A'$, we get the coefficient is:
\begin{align*}
|a_{m-1}| &= | (-1)^{|s_A|-2} e_A' + (-1)^{|s_A|-1} (|s_A|-1) |\\
		  &= | (-1)^{|s_A|-2} (e_A' - |s_A|+1)\\
		  &= e_A' - |s_A|+1
\end{align*}
and we are done with part (1).

To prove part  (2)  we can apply the argument of (1) to the diagram  $D^{*}$ as in the proof of Lemma \ref{lem:firstCoef}. This finishes the proof of the lemma and that of Theorem ~\ref{prop:JonesPoly}.

\end{proof}
\vskip 0.02in

The notion of geometric adequacy is well suited for the connections of skein invariants to geometry of the link complement that we explore in this paper. 
Recently, Boden, Karimi and Sikora  have considered link diagrams  on a surface $F=F\times\{0\}$ inside in thickened surfaces $F\times [-1, 1]$
and they also defined notions of $A$-adequacy  and $B$-adequacy. We now compare their definitions to ours.
In the terminology of \cite{ThickTaitConj}
a link diagram $D$ on $F$ is called $A$-adequate if for any state $s$ that differs only by a single resolution
from $s_A$ we have the following: either $|s_t|\leq |s_A|$ or the number of non-contractible state circles in $s$ is different than this in $s_A$. One defines $D$ being $B$-adequate in a similar way.

\begin{lem}\label{lem:comp}
Suppose that $D=\pi(L)\subset F\subset M$  is a link diagram such that
all the state circles in $s_A$ are contractible on $F$. Then, if $D$ is geometrically $A$-adequate, it is also $A$-adequate.
\end{lem}

\begin{proof}
To show that $D$ is $A$-adequate, we must show that, for any state $s$ adjacent to $s_A$, either $|s_t|\le |s_A|$ or $s$ and $s_A$ have a different set of non-contractible loops. There are two ways a state change $s_A\to s$ can increase  $|s_t|$: either we split a contractible  state circle into two such circles, or we merge two essential circles in $s_A$ into a contractible  circle. By assumption, as $s_A$ has only contractible circles,  we must split a single state circle into two. However, in order to split a state circle, we must have an edge of $G_A$ connecting that state circle to itself. As $G_A$, by assumption, has no 1-edge loops, this cannot happen, and so we are done.
\end{proof}

The converse of  Lemma ~\ref{lem:comp}  is not true.
The diagram $D$ of Figure \ref{alldisks}	 is $A$-adequate in the sense of ~\cite{ThickTaitConj}  but is not geometrically $A$-adequate: indeed, while the state circles in $s_A$ are contractible,  each of the four
states that are obtained from $s_A$ by a single by a single resolution change contains non-contractible circles. Thus $D$ is  $A$-adequate. However, $G_A$ contains 1-edge loops,
hence $D$ is not    geometrically $A$-adequate.

\begin{figure}
	\centering
		\def\svgwidth{.65\columnwidth}
		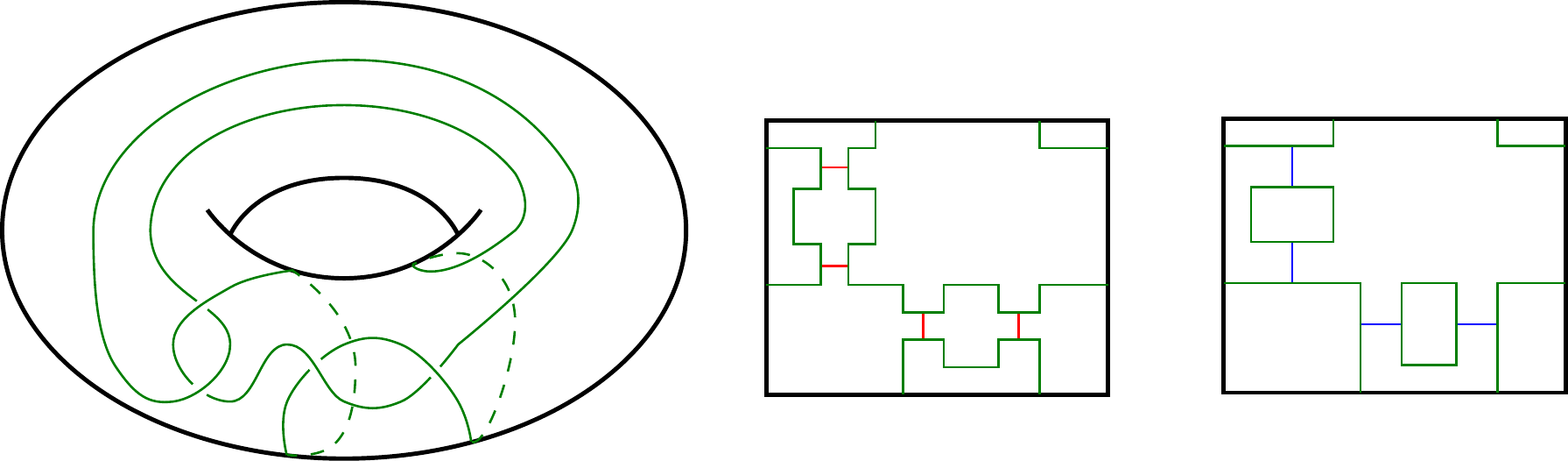

		\caption{A  link diagram $D=\pi(L)$ (left) 
		together with the  graphs $G_A$ (center) and $G_B$ (right). All the state circles in $s_A$ and $s_B$ are contractible on the torus. The diagram $D$  geometrically $B$-adequate but not
		geometrically $A$-adequate. In fact all the four edges of $G_A$ are 1-edge loops.}
\label{alldisks}	
		
\end{figure}

\medskip

\section{Guts of surfaces and Kauffman brackets}
In this section we focus on links that admit alternating projections on surfaces in 3-manifolds. We find that under suitable diagrammatic conditions such links are geometrically adequate  on one hand, and on the other hand they fit into the class of weakly generalized alternating links studied in \cite{WGA}. This will allow us to combine our work in the last section with the geometric techniques of 
\cite{WGA} and prove
 Theorems  \ref{thm:GutsAndJones}  and  \ref{thm:volume}.

\subsection{Reduced alternating link projections}
Suppose $M$ is a compact, orientable, irreducible 3-manifold with empty or incompressible boundary and  $F$ a closed, connected, orientable  surface in $M$.
Next we introduce several properties and definitions about projections of links
$L\subset F\times [-1, \ 1]\subset M $ on $F$.
 Some of these  properties are  directly quoted from ~\cite{WGA} and others are suitably adapted to fit our purposes better.

\begin{defn} \label{prime}
A  link diagram $\pi(L)\subset F$ is \textit{prime} if, whenever a disk $D\subset F$ has $\bndry D$ intersecting $\pi(L)$ transversely exactly twice, then either:
\begin{itemize}
	\item $F=S^2$, and either $\pi(L)\cap D$ is a single arc or $\pi(L)\cap (F\setminus D)$ is; or
	\item $F$ has positive genus, and $\pi(L)\cap D$ is a single arc. 
\end{itemize}
\end{defn}

\begin{defn} \label{defn:reducible}
We say a link diagram $\pi(L)\subset F$ is \textit{reduced alternating} if
\begin{enumerate}
\item each component of $L$ projects to at least one crossing in $\pi(L)$, 
\item  $\pi(L)$ is prime and alternating on $F,$  and 
\item for every essential,  simple closed curve $\gamma$ on $F$ that intersects $\pi(L)$ at exactly two points near a crossing, one of the two sub-arcs of  $\pi(L)$ with endpoints on $\gamma$
	contains no crossings of $\pi(L)$.
\end{enumerate} 
\end{defn}

We note for Definition \ref{prime} the authors   of \cite{WGA} use the term ``weakly prime".
We also note that Definition \ref{defn:reducible} is different than the definition of reduced diagrams given in  ~\cite{WGA}, in that condition condition (3) is not required in their definition. For alternating projections of $F=S^2$ in $M=S^3$ the two definitions are equivalent.

Given an alternating  link projection  $\pi(L)\subset F\subset M$ for each crossing of $\pi(L)$ we can label the four regions around it by the letters $A$ and $B$ in an alternating fashion. This is done so that the two opposite regions of the crossing that are 
merged during the $A$ splitting are labeled by $A$. Similarly, the two opposite regions of the crossing that are 
merged during the $B$ splitting are labeled by $B$. This way the corners of every region of $F\setminus \pi(L)$ receive the label $A$ or $B$.
\begin{defn}\label{check} With the notation and setting as above, we will say that the link diagram $\pi(L)$ is 
\textit{checkerboard colorable} if for every region $R$ of $F\setminus \pi(L)$ the letters at all corners of $R$ are the same.  Thus every region of a checkerboard  colorable diagram is labelled by $A$ or $B$.
\end{defn}

Next we recall two complexity functions for link diagrams $\pi(L)\subset F\subset M$ from  \cite{WGA}.

\begin{defn}\label{ourreduced}
The \textit{edge representativity} $e(\pi(L),F)$ is the minimum number of intersections between $\pi(L)$ and any essential curve on $F$. If there are no essential curves, then we say $e(\pi(L),F)=\infty$.

The \textit{representativity} $r(\pi(L),F)$ is the minimum number of intersections between $\pi(L)$ and the boundary of  any compressing disk for $F$. If there are no compressing disks for $F$, then we say $r(\pi(L),F) =\infty$.
\end{defn}

As an example to clarify the definitions above we discuss the alternating link diagram $\pi(L)$ of Figure \ref{alldisks}	viewed on a standard Heegaard torus  $F=T^2$  in $S^3$.
We have $e(\pi(L),F)=r(\pi(L),F)=2$ and the diagram is checkerboard colorable,  prime and all the regions of $T^2\setminus \pi(L)$ are disks. However, $\pi(L)$  is not reduced in the sense of Definition \ref{defn:reducible}.
The next proposition  shows this phenomenon doesn't happen when $e(\pi(L), F)>2$. 

\begin{prop} \label{prop:ReducedEdgeRep}
Let $\pi(L)\subset F\subset M$ be an alternating link diagram such that $\pi(L)$ is checkerboard colorable and all the regions of $F\setminus \pi(L)$ are disks. 
Then $\pi(L)$ is reduced if and only if 
\begin{itemize}
\item $\pi(L)$ is  prime,
\item each component of $L$ projects to at least one crossing in $\pi(L),$ and
\item the edge representativity satisfies
$e(\pi(L),F)>2$.
\end{itemize}
 \end{prop}

\begin{proof}
First, suppose $e(\pi(L),F)>2$ and $\pi(L)$ is  prime. We only need to prove that $\pi(L)$ satisfies part (3) of Definition \ref{defn:reducible}. So suppose $\gamma$ is a simple closed curve on $F$ intersecting $\pi(L)$ exactly twice. Because $e(\pi(L),F)>2$, we know that $\gamma$ cannot be essential, and so must bound a disk $E$ on $F$. If $F=S^2$, then, as $\pi(L)$ is prime, either $\pi(L)\cap E$ is a single arc without any crossings, or $\pi(L)\cap (F\setminus E)$ is. If $F\neq S^2$, then $\pi(L)\cap E$ is a single arc. In either case, we have one of the sub-arcs of $\pi(L)$ with endpoints on $\gamma$ contain no crossings, and so we are done.

Now suppose $\pi(L)$ is reduced alternating. Then we already know that $\pi(L)$ is prime, and each component of $L$ projects to at least one crossing in $\pi(L).$ We  need to show $e(\pi(L),F)>2$. Suppose not. As $D$ is checkerboard colorable, we must have $e(\pi(L),F)$ be an even number. If $e(D,F)=0$, then there is a region of $F\setminus \pi(L)$ that contains an essential curve. This would mean we have a non-disk region, and so cannot happen.

If $e(\pi(L),F) = 2$, then we can find some essential closed curve $\gamma$ intersecting $\pi(L)$ exactly twice. As $\pi(L)$ is reduced, this means that one of the two sub-arcs of $\pi(L)$ with endpoints on $\gamma$ must contain no crossings. Call this sub-arc $\ell$. We also have $\gamma$ split into two sub-arcs, $\gamma_1$ and $\gamma_2$. There are four cases to consider. 
First, we could have $\ell\cup \gamma_i$ bound a disk on $F$, for some $i=1,2$. We can use this disk to homotope $\gamma$ off of $D$, and so become an essential curve interesting our knot zero times, a contradiction.
Second, we could have $\ell$ form a single component. If it does, we then have a component of $L$ with no crossing in $D$, 
contradicting the assumption that each component of $L$ projects to at least one crossing on $F$, included in the definition of reduced.
Third, we could have $\ell \cup \gamma_i$ essential and parallel to all of $\gamma$, for some $i$, say $i=1$. But then $\ell \cup \gamma_2$ is homotopically trivial, and so we are in the first case, and get a contradiction.
Finally, we could have $\ell\cup\gamma_i$ essential and not parallel to $\gamma$. As $\pi(L)$ is checkerboard colorable and $\ell$ contains no crossings, everything to one side of $\ell$ must be the same color. But then we have a region adjacent to itself across a knot arc, contradicting $D$ being checkerboard colorable.
In any of the cases where $e(\pi(L),F) = 2$, we contradict one of our assumptions, and so it cannot happen. So then we are left with $e(\pi(L), F)>2$, and we are done.
\end{proof}

\begin{defn}
Following \cite{WGA} we say a link diagram $\pi(L)\subset F$ is \textit{weakly generalized alternating} if is prime, checkerboard colorable, alternating,
and the representativity satisfies $r(\pi(L), F)\ge 4$.
\end{defn}

We have the following:

\begin{cor} \label{reducedisWGA} A reduced alternating diagram  $\pi(L)\subset F$ that has disk regions and is checkerboard colorable is also weakly generalized alternating.
\end{cor}
\begin{proof} We only need to check that the condition that $r(\pi(L), F)\ge 4$.
By Proposition \ref{prop:ReducedEdgeRep}, we have $e(\pi(L), F)>2$ which since the diagram is checkerboard colorable 
with disk regions, implies $e(\pi(L), F)\ge 4$.
Since any curve on $F$  bounding a compression disk  is also essential we are done.
\end{proof}

\begin{figure}[H]
	\centering
		\def\svgwidth{0.7\columnwidth}
		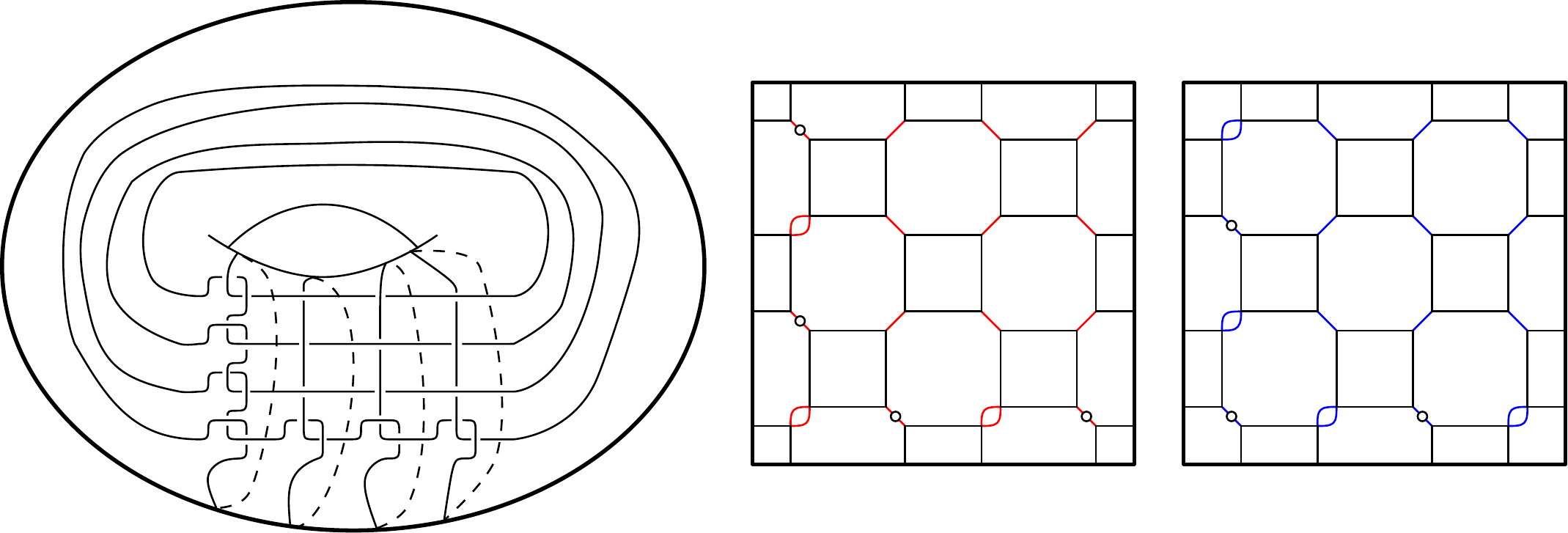

		\caption{ A reduced, checkerboard colorable, alternating diagram  $\pi(L)$ with disk regions, on a torus. In the terminology of \cite{WGA} it is weakly generalized alternating.
		Both $G_A$ \textit{(left)} and $G_B$ \textit{(right)} have no 1-edge loops making $\pi(L)$  geometrically adequate.}
\label{exampleadequate}
\end{figure}

Our next lemma together with Corollary \ref{reducedisWGA}  allow us to relate our work in Section 2 to the work of \cite{WGA} on weakly generalized alternating links.

\begin{lem}\label{conditions} 
Suppose  $\pi(L)$ is an alternating diagram on a projection   surface $F$ of genus at least 1 in a 3-manifold $M$. 
Suppose that $\pi(L)$ is reduced,  checkerboard colorable
and all regions of $F\setminus \pi(L)$ are disks. Then $\pi(L)$ is geometrically adequate.
\end{lem}
\begin{proof}

First, as $\pi(L)$ is alternating on $F$ and all regions of $F\setminus D$ are disks, we have that $s_A$ and $s_B$ must have only contractible circles. 

We need to show that  $G_A$ and $G_B$ have no 1-edge loops.
The  proof is the same for both $G_A$ and $G_B$, so we will focus on $G_A$. Suppose $\pi(L)$ is as in the statement of the lemma, but $G_A$ has at least one 1-edge loop, $\ell$. Then $\ell$ connects a state circle to itself, and $\ell$ crosses $\pi(L)$ exactly twice at a crossing. We may then find some simple arc $\gamma$ in the state circle connecting the two endpoints of $\ell$. But then $\gamma\cup \ell$ is a simple closed curve intersecting $\pi(L)$ exactly twice. By Proposition \ref{prop:ReducedEdgeRep}
we have $e(\pi(L), F)\ge 4$. Thus $\gamma\cup \ell$ must be contractible on $F$. 

By homotoping  $\gamma\cup \ell$  we can get two such simple closed curves, one with the crossing to the left of the curve, and the other with the crossing to the right. One of them would bound a disk such that
the existence of the crossing corresponding to $\ell$ violates the primeness of $\pi(L)$.
Thus $\pi(L)$ must be geometrically $A$-adequate. See Figure \ref{exampleadequate} for an example.
\end{proof}

\subsection{Twist number relations and invariance} 
A \textit{twist region} of an alternating projection $\pi(L)\subset F$  is either 
 a string of bigons of $\pi(L)$  arranged vertex to vertex that is maximal in the sense that no larger string of bigons contains it, or a single crossing adjacent to no bigon.

\begin{defn}
An alternating diagram  $\pi(L)\subset F$ is called \textit{twist reduced} if whenever there is a disc $D\subset F$ such that $\partial D$ intersects $\pi(L)$ exactly four times adjacent to two crossings, then one of the following holds:
\begin{itemize}
\item $D$ contains a (possibly empty) sequence of bigons that is part of a larger twist region containing the two crossings, or
\item $F\setminus D$ contains a disc $D'$, with $\partial D'$ intersecting $\pi(L)$ four times adjacent to the same two crossings as $\partial D$, 
and $D'$ contains a string of bigons that forms a larger twist region containing the original two crossings. 
See Figure \ref{Fig:TwistReduced}.

\end{itemize}

The \textit{twist number} $ t_F(\pi(L))$ of a diagram is the number of twist regions in a  twist-reduced diagram.
\end{defn}

\begin{lem}
\label{lem:twistNumber} Suppose that $\pi(L)$ is a twist-reduced, reduced alternating diagram with twist number $ t_F(\pi(L))$  on a projection surface $F\subset M$ of genus at least 1.
Suppose that $\pi(L)$ is  checkerboard colorable
and all regions of $F\setminus \pi(L)$ are disks. Then, we have
\[|a_{m-1}| + |b_{n+1}| - 2 = t_F(\pi(L)) - \chi(F),\]
where $a_{m-1}$ and $b_{n+1}$ are the second and the penultimate coefficient of the polynomial $J_0(\pi(L))$.
\end{lem}

\begin{proof}

By Lemma \ref{conditions}  $\pi(L)$ is geometrically adequate;  the state graphs $G_A$ and $G_B$ have no 1-edge loops.
Suppose that $\pi(L)$ has $c$ crossings. First note that the twist number is
\[ t_F(\pi(L)) = c - (c-e'_A) - (c-e'_B) = e'_A + e'_B - c.\]

By definition, crossings that correspond to twist region of $\pi(L)$ correspond to edges of $G_A$ or $G_B$ that are parallel; every pair bounds bigon on $F$. Call a twist region  of $\pi(L)$ an $A$ (or $B$) twist region if, in $G_A$ (or $G_B$), all crossings of the twist region are represented by edges such that every pair bounds a bigon on $F$. 
Then note that $c-e'_A$ is exactly the number of edges in $G_A$ that aren't in $G'_A$, and so counts the number of crossings that are in an $A$ twist region except for one for each such twist region (that is represented in $G'_A$). Likewise, $c-e'_B$ is the number crossings in $B$ twist regions minus one for each such twist region. Then $(c-e'_A) + (c-e'_B) = c -  t_F(\pi(L))$.

Next, note that 
\[|s_A| + |s_B| =  c + 2 - 2g(F) = c + \chi(F).\]
Putting these together, along with Lemmas ~\ref{lem:firstCoef} and ~\ref{lem:secondCoeff}, we get:
\begin{align*}
|a_{m-1}| + |b_{n+1}| - 2 &= e'_A + e'_B - |s_A| - |s_B|\\
	&= ( t_F(\pi(L)) + c) - (c + \chi(F)) \\
	&=  t_F(\pi(L))- \chi(F),
\end{align*}
which finishes the proof of the lemma.
\end{proof}

\begin{figure}
\def\svgwidth{0.7\columnwidth}
  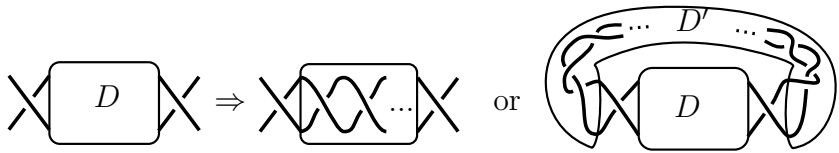
  \caption{A  twist reduced diagram. Figure modified from \cite{ALSVB}.}
  \label{Fig:TwistReduced}
\end{figure}

It is known that the twist number of a reduced, twist-reduced alternating projection $\pi(L)$ on a 2-sphere in $S^3$ is an isotopy invariant of $L$. This has been proven in two ways. Firstly,  it follows from work of Dasbach and Lin  \cite{VTJPAK}  showing this twist number can be obtained from the Jones polynomial of $L$. Secondly, it follows from the Tait flyping conjecture proved in Menasco and Thistlewaite \cite{ClassAltLinks}, which shows that any two reduced, prime alternating link diagrams are related by a series of flypes. 
Following the approach of \cite{VTJPAK},
we have a generalization of twist number invariance for alternating links in thickened surfaces.
\invariant*
\begin{proof}
By Lemma \ref{lem:twistNumber}, $|a_{m-1}| + |b_{n+1}| - 2+ \chi(F) = t_F(\pi(L))$.
Since $|a_{m-1}|, |b_{n+1}|$ isotopy invariants of $L$ in $F\times[-1,1]$ (Proposition \ref{isinvariant}), the conclusion follows.
\end{proof}

 The Tait flyping conjecture is unknown for links in thickened surfaces. Hence the second method of deducing invariance of the twist number  is not currently available. However, Boden, Karimi, and Sikora were able to show the first two Tait conjectures by proving that, for reduced alternating diagrams in thickened surfaces, the crossing number and the writhe are invariants \cite{ThickTaitConj}. 

In general the twist number of weakly generalized alternating knots is not an invariant. Howie has produced weakly alternating projections of the same knot on a Heegaard torus of in $S^3$
with different twist numbers (e.g. the knot $9_{29},$ is one example) \cite{Howie}. On the other hand no such examples are known for weakly alternating projections on incompressible surfaces.
In the view of this and Corollary 1.3 we ask the
following:

\begin{que}
Let $M$ be a 3-manifold that is closed or has incompressible boundary and $F\subset M$ an incompressible surface.
Suppose that
$\pi(L)$ is  a reduced, twist-reduced, checkerboard colorable, alternating diagram on $F$ where all the regions of $F\setminus \pi(L)$ are disks. 
Is $t_F(\pi(L))$ is an invariant of the isotopy type of $L$ in $M$?
\end{que}

As Howie's examples take place on the compressible Heegaard torus in $S^3$, these do not give an answer to this question.

\subsection{Guts and Kauffman bracket} Here we will prove Theorem ~\ref{thm:GutsAndJones} stated in the introduction. In fact we prove a more general result (Theorem \ref{thm:Gutsgeneral}) in which the assumption that $F$ is incompressible ($r(\pi(L),F)=\infty$) is relaxed to  $r(\pi(L),F)>4$.

Suppose that $D=\pi(L)$ is a weakly generalized alternating diagram on a surface $F\subset M$ such that the regions of $F\setminus \pi(L)$ are all disks.
The projection gives rise to two spanning surfaces of $L$,  the checkerboard surfaces that we will denote by $S_A=S_A(D)$ and $S_B=S_A(D)$. Our convention will be that $S_A$ is constructed
by attaching half twisted bands to the disks bounded by the state circles  $s_A(D)$, where we attach a half twisted band for each crossing of $D,$
so that the band retracts onto the corresponding edge of the graph $G_A$ and the surface $S_A$ retracts to $G_A$.
Similarly we define $S_B$ that retracts onto $G_B$.
See Figure ~\ref{fig:surfaces}.

\begin{figure}[ht]
\includegraphics[scale=.8]{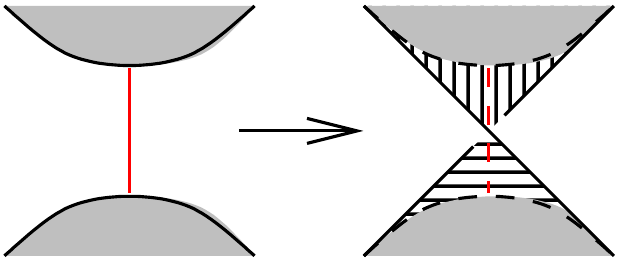}
\hspace{2cm}
\includegraphics[scale=.8]{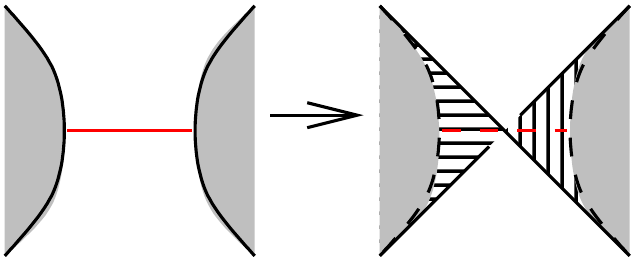}
\label{resolve}
\caption{The construction of $S_A$ and $S_B$. The red lines indicate the edge of $S_A$ and $S_B$ the corresponds to the bands shown.}
\label{fig:surfaces}
\end{figure} 

By \cite[Theorem 3.19]{WGA} the surfaces $S_A$ and $S_B$ are $\pi_1$-essential in the complement of $X=M\setminus L$.
Let $M_{A} = X\cut S_A := X\setminus N(S_A)$ and Let $M_{B} = X\cut S_B := X\setminus N(S_B)$. Recall also that $a_m, a_{m-1}, b_{n+1},$ and $b_n$ are the first two and last two coefficients, respectively, in the polynomial $J_0(\pi(L))$.

\begin{thm}\label{thm:Gutsgeneral} 
Let $M$ be a 3-manifold that is closed or has incompressible boundary and $F\subset M$ a projection surface such that
that $M\setminus N(F)$ is atoroidal
and $\bndry$-anannular.  
Let $\pi(L)$be  a reduced,   alternating diagram on $F$ that is   twist-reduced with twist number $t_F(\pi(L))$.
Suppose that  $\pi(L)$  is checkerboard colorable and all the regions of $F\setminus \pi(L)$ are disks. Suppose, moreover, that  $F$ has genus at least 1 and the representativity satisfies  $r(\pi(L),F)>4$. 
Then,

\begin{enumerate}
\item we have $\chi(\guts(M_A)) = 1-|a_{m-1}|+ \frac{1}{2}\chi(\bndry M)$,
\item we have $\chi(\guts(M_B)) =  1-|b_{n+1}| + \frac{1}{2}\chi(\bndry M)$,
\item we have $t_F(\pi(L))=|a_{m-1}| + |b_{n+1}| - 2+\chi(F).$
\end{enumerate}
\end{thm}

Let us first explain how to  deduce Theorem ~\ref{thm:GutsAndJones}: as discussed earlier, in the case that $F$ is incompressible in $M$ we  have $r(\pi(L),F)=\infty$.
Thus, in particular, the condition $r(\pi(L),F)>4$ is satisfied and Theorem ~\ref{thm:GutsAndJones} is a special case of Theorem \ref{thm:Gutsgeneral}. \qed

\begin{proof} First note that by Lemma ~\ref{conditions},   $D=\pi(L)$ is geometrically $A$-adequate and geometrically $B$-adequate.
We will give the proof for part (1) and $M_A$. The proof works the same, after swapping $S_A$ and $S_B$ to give part (2). Finally, part (3) follows from Lemma
~\ref{lem:twistNumber}.

The graph $G_A$ gives a cellular decomposition of the surface $F$. The number of $0$-cells is the number of the vertices $G_A$, denoted by $|s_A|$, and  the number of 1-cells is the  number of edges $e_A=c(\pi(L))$. The number
of 2-cells is the number of complementary regions of $G_A$ which is the same as the number $|s_B|$ of vertices of $G_B$. If we consider $\pi(L)$ as a 4-valent graph on $F$ we can label the components
of $F\setminus \pi(L)$ by $A$ or $B$ according to whether  they correspond to a vertex of $G_A$ or $G_B$. We will refer to these as $A$-regions and $B$-regions, respectively.
Now let $|s'_B|$ denote the number of non-bigon $B$-regions and recall that $e'_A$ denotes the number of edges in the reduced graph $G'_A$.
We have,

\begin{equation}
\label{eqn:EulerCharResult}
\chi(F) = |s_A| - e_A+ |s_B| = |s_A| - e'_A+ |s'_B|,
\end{equation}

where the second equality follows since, by definition and the fact that $D$ is twist-reduced, the number of edges we remove from $G_A$ to obtain $G'_A$ is exactly the number of bigon $B$-regions.
Equation (\ref{eqn:EulerCharResult}) gives the following.

\begin{equation}
\label{eqn:EulerCharResult1}
 \chi(F) - |s'_B| = |s_A| - e_A'.
\end{equation}

Since, as we mentioned above, $D$ is geometrically $A$-adequate by Theorem \ref{prop:JonesPoly} we have

\begin{equation} 
|s_A| - e_A' =1-|a_{m-1}|= |a_{m}|-|a_{m-1}|.
\label{eqn:coefsrel}
\end{equation}

By Corollary \ref{reducedisWGA},   $D$ is weakly generalized alternating. Now we can apply ~\cite[Theorem 6.6]{WGA}
to $D$ to conclude that \begin{equation}
\chi(\guts(M_A)) = \chi(F)  + \frac{1}{2} \chi(\bndry M) - |s'_B|.
\label{eqn:HPGuts}
\end{equation}

Now   combining Equation (\ref{eqn:HPGuts}) with Equations (\ref{eqn:EulerCharResult1}) and (\ref{eqn:coefsrel}), we 
get

$$\chi(\guts(M_A)) =1-|a_{m-1}|+ \frac{1}{2} \chi(\bndry M),$$
which is part (1) of the theorem.

We will now sketch the proof of Equation \ref{eqn:HPGuts}, referring the reader to \cite{WGA} for precise definitions and details. We do this not only for reasons of completeness, but because it is interesting
to see the correspondence between the combinatorics in the calculation of $|a_{m-1}|$ from the proof of \ref{prop:JonesPoly} and these involved in the calculation of $\chi(\guts(M_A))$.
On one hand, edges that  are parallel on $G_A$ (i.e that co-bound a disk on $F$) don't contribute due to cancellations  Kaufman state sum expression of $|a_{m-1}|$.
On the other hand  strings of parallel edges on $G_A$  correspond components of $I$-bundle pieces of the JSJ-decomposition$M_A$ and they don't contribute $\chi(\guts(M_A))$.
  
Setting  $\tilde{S}_A = \bndry N(S_A)$,   the  \textit{parabolic locus} $P$ is $\bndry M_A \cap \bndry N(L)$. Considering $\pi(L)$ as a 4-valent graph on $F$, the authors in \cite{WGA} define a
 \textit{chunk decomposition} of $M_A$: this decomposes $M_A$ into two 
 compact, oriented, irreducible 3-manifolds with boundary, say  $C_1, C_2$,  each containing a copy $F$ as a boundary component (and possibly more boundary components coming from $\bndry M$). 
 The component of $\partial C_i$ that corresponds to $F$ comes equipped with a checkerboard coloring with the regions of $F~\setminus ~\pi(F)$ called  \textit{faces}.
 The  chunks are glued together along the $B$ labeled faces. The decomposition generalizes  previously known polyhedral decompositions constructed from  alternating and adequate link projections
 in $S^3$  (see, for example,  \cite{Guts} and references therein). 
  Even though the chunks are not simply connected,
 \cite{WGA}  shows that techniques that were used for polyhedral decompositions generalize and adapt in the setting of chunks.

Recall that $M\setminus N(F)$ is atoroidal
and $\bndry$-anannular. 
By the annulus  version of  JSJ-decomposition one can cut $M_A$ along a collection of essential annuli that are disjoint from $P$ into $I$-bundles, Seifert fibered pieces and hyperbolic pieces which are the ones that form the guts. Seifert fibered pieces turn out to be solid tori and as such they don't contribute to the Euler characteristic computation.

Let $R$ be an essential annulus  in $M_A$, disjoint from $P$ with $\bndry R \subset \tilde{S}_A$.  Such an annulus $R$ is either
 \textit{parabolically compressible} or not, in which case they are called  \textit{parabolically incompressible}.

If $R$ is parabolically incompressible, then \cite[Lemma 6.9]{WGA} argues that $F$ must be a 2-sphere,  contradicting the assumption of $F$ having genus at least 1
 
Suppose now that $R$ is parabolically compressible. This means that
 there is a disk $D$ with interior disjoint from $R$, with $\bndry D$ meeting $R$ in an essential arc $\alpha$ on $R$, and with $\bndry D\setminus \alpha$ lying on $\tilde{S}_A\cup P$, with $\alpha$ meeting $P$ transversely exactly once. 
If we do surgery along such a disk we obtain an \textit{essential product disk}: these are 
disks meeting $P$ transversely exactly twice, with boundary otherwise on $\tilde{S}_A$. Such disks are known to correspond to  $I$-bundle components of above mentioned JSJ-decomposition (see \cite[Definition 4.5]{Guts} or \cite[Definition 6.7]{WGA}).

Now let us look at an essential product disk $E$ caused by surgering $R$.
If it meets $S_B$, then $S_B$ cuts $E$ into sub-rectangles $E_1,\cdots, E_n$. By looking how such rectangles must sit in the diagram and in the chunk decomposition, one can show that $E$ must be boundary parallel, a contradiction to $E$ being essential. 

However, if $E$ doesn't run through $S_B$, then $\bndry E$ must meet the chunk in two $A$-faces and two $B$-faces, and so $\bndry E$ must meet $P$ exactly four times. Such an $E$ is parallel into $F$. However, as $D$ is twist reduced, this implies $\bndry E$ contains a series of $B$-bigons.
\vskip 0.02in

{\underline{{\it {Case 1:}}}}\ First suppose that we don't have $B$-regions that are bigons. Then $\guts(M_A) = M_A$. Recall that $M_A$ is obtained by $C_1$ and $C_2$, where we glue these chunks together along $B$-labeled faces. Then, as $\chi(C_i) = \frac{1}{2}\chi(\bndry C_i)$, we must have
\begin{align*}
\chi(C_1) &= \frac{1}{2}\chi(F) + \frac{1}{2}\chi(\bndry M|_{C_1})\\
\chi(C_2) &= \frac{1}{2}\chi(F) + \frac{1}{2}\chi(\bndry M|_{C_2}).
\end{align*}
Gluing the chunks together along white faces will add their Euler characteristics together, and subtract one for every $B$-face we glue along. As there are no white bigons, we glue along $|s_B|=|s'_B|$ such faces, and so:
\[\chi(\guts(M_A)) = \chi(F) + \frac{1}{2}\chi(\bndry M) - |s'_B|.\]

{\underline{{\it {Case 2:}}}}\ If $F\setminus \pi(L) $ has  $B$-regions that are bigons, then each such bigon will form a quad, with two
sides on $P$ and two sides on $\tilde{S}_A$. These will give essential product disks and thus $I$-bundle parts.
The existence of $I$-bundles leads to parabolically compressible annuli and, as mentioned above, to essential product disks. All the essential product disks parabolically compress
to the strings of the ones corresponding to bigons   (see\cite[Figure 4.2]{Guts}). Surgering along one of these basic essential product disks  increases the Euler characteristic of the $I$-bundle sub-manifold by one and it doesn't change the guts.
After we remove all $B$-bigons of $\pi(L)$, we have replaced each $B$ twist region by a single crossing and we have eliminated all the $I$-bundle components. This has modified $S_A$ into a new surface $S'_A$
and $\guts(M_A)$ is the same as the guts of $S'_A$. But now we have no $B$-bigon regions left, and the $B$-regions of the new link projection are exactly the non-bigon  $B$-regions of 
$\pi(L)$, that is exactly $|s'_B|$. 
\end{proof}
\medskip

\section{Relations with hyperbolic geometry invariants}

Let $D=\pi(L)$ be a reduced, alternating link diagram on a surface $F\subset M$, such that $M\setminus L$ is hyperbolic.
In this section we show that the skein theoretic quantities $|a_{m-1}(\pi(L))|$, $ |b_{n+1}(\pi(L))| $ provide bounds on the volume of the complement $M\setminus L$.
The relations with volume come from two sources: first, by a result of Agol, Storm and Thurston \cite{AST}
the negative Euler characteristic of the guts of an essential surface in a hyperbolic 3-manifold $M$ bounds linearly the volume of $M$ from below.  
We will apply this result to the surfaces $S_A$, $S_B$ associated to projections of weakly generalized alternating links.
Second, by work of Kalfagianni and Purcell \cite{ALSVB},  if $F$ is Heegaard torus or $M$ is a thickened surface,   
the twist number of weakly generalized alternating projections provides two sided bounds of their volume. 

We prove the following theorem which, in particular, implies Theorem \ref{thm:volume} stated in the Introduction.

\begin{thm}
\label{thm:volumegeneral} Let $M$ be a 3-manifold that is closed or has incompressible boundary and $F\subset M$ a projection surface such that
that $M\setminus N(F)$ is atoroidal
and $\bndry$-anannular.  Let $D=\pi(L)$ be a reduced and  twist-reduced   alternating diagram on $F,$ that is checkerboard colorable and all the regions of $F\setminus D$ are disks.
Suppose, moreover, that  $F$ has genus at least 1 and the representativity satisfies  $r(D,F)>4$. 
Then $L$ is hyperbolic
and we have
$$\vol(M\setminus L) \ \geq v_8\  
{\rm{max}}\{ |a_{m-1}|, \   |b_{n+1}| \}-1-\frac{1}{2}\chi(\bndry M),$$
where $a_{m-1},   b_{n+1}$ are the second and penultimate  coefficient of the polynomial $J_0(\pi(L))$, and  $v_8=3.66386\cdots$ is the volume of a regular ideal octahedron.
\end{thm}
\begin{proof}

By Corollary \ref{reducedisWGA},   $D$ is weakly generalized alternating. 
Let $S_A$ and $S_B$ denote the checkerboard surfaces of the projection.
 By ~\cite[Theorem 1.1]{WGA},  $S_A$, $S_B$ are $\pi_1$-essential in $X=M\setminus L$, and  $X$ is hyperbolic.
By cutting the link complement along $S_A$ and $S_B$ we obtain manifolds  $M_A = X \cut S_A$ and $M_B = X \cut S_B$, respectively. 
By  ~\cite[Theorem 9.1]{AST} we have
$$\vol(M\setminus L) \ \geq  -  v_8 \  \chi(\guts(M_A)), \ \ {\rm and} \ \  \vol(M\setminus L) \ \geq  -  v_8 \  \chi(\guts(M_B)).$$

Since we assumed that $\pi(L)$ is reduced by Lemma \ref{conditions}, it is geometrically $A$-adequate and $B$-adequate. By Theorem \ref{thm:Gutsgeneral}
we have

$$ \chi(\guts(M_A))=1- |a_{m-1}|+\frac{1}{2}\chi(\bndry M) \  \ {\rm and} \  \  \chi(\guts(M_B))=1- |b_{n+1}|+\frac{1}{2}\chi(\bndry M).
$$

Thus we obtain

$$\vol(M\setminus L) \ \geq    v_8 \left( |a_{m-1}|  - 1 -\frac{1}{2}\chi(\bndry M)\right),$$
 and
 
 $$\vol(M\setminus L) \ \geq   v_8 \left( |b_{n+1}| - 1 - \frac{1}{2}\chi(\bndry M)\right),$$
and the result follows.
\end{proof}

To see how Theorem \ref{thm:volume}   follows note that if $F$ is incompressible the hypothesis $r(D,F)>4$ is satisfied.

Next we discuss two  special cases where  the quantity
$  |a_{m-1}| + |b_{n+1}|-2 ,$
of Theorem ~\ref{thm:volumegeneral} also provides an upper bounds of the volume.
The first result concerns weakly generalized alternating knots on a Heegaard torus.

\begin{cor}\label{Cor:TorAltUpperBound} Let $F$ be a Heegaard torus $F$ in $M=S^3$, or in a lens space $M=L(p,q)$.
Let $D=\pi(L)$ be a reduced and twist-reduced   alternating diagram on $F,$ that is checkerboard colorable and all the regions of $F\setminus D$ are disks.
Suppose, moreover that the representativity satisfies  $r(D,F)>4$. 
Then $M\setminus L $ is hyperbolic, and
\[  \frac{v_8}{2}\  \left( |a_{m-1}| + |b_{n+1}|-2\right)  \leq \vol(M\setminus L) < 10\, v_4\cdot \left( |a_{m-1}| + |b_{n+1}|-2\right),\]
where $v_4= 1.01494\dots$ is the volume of a regular ideal tetrahedron.
\end{cor}
\begin{proof}By Corollary \ref{reducedisWGA} the projection $\pi(L)$ is weakly generalized alternating.  Hyperbolicity follows from
\cite[Theorem 1.1]{WGA}. By ~\cite[Corollary 1.5]{ALSVB}, which also relies on \cite{WGA} for the lower bound,  we have

\[  \frac{v_8}{2}\  t_F(\pi(L)) \leq \vol(M\setminus L) < 10\, v_4\cdot t_F(\pi(L)),\]
where $ t_F(\pi(L))$ is the twist number of $\pi(L)$. 
By Lemma ~\ref{lem:twistNumber}
$t_F(\pi(L))=|a_{m-1}| + |b_{n+1}| - |a_{m}| - |b_{n}|=|a_{m-1}| + |b_{n+1}| -2$ and the result follows.
\end{proof}

Our second result is Theorem ~\ref{Thm:FxIUpperBound} which we now prove.
\begin{proof}By Corollary \ref{reducedisWGA} the projection $\pi(L)$ is weakly generalized alternating. Hyperbolicity follows from
\cite[Theorem 1.1]{WGA}, where when $F\neq T^2$ the hyperbolic structure is chosen so that non-torus boundary components of $M\setminus L$ are totally geodesic.
By ~\cite[Theorem 1.4]{ALSVB}, which relies on \cite{WGA} for the lower bound, we have
\begin{equation}
\frac{v_8}{2}\ t_F(\pi(L)) \leq \vol(Y-K) < 10\, v_4\cdot t_F(\pi(L)), 
\label{eq:tn}
\end{equation}
if $F=T^2$, and we have

\[  \frac{v_8}{2}\ ( t_F(\pi(L))- 3\chi(F)) \leq \vol(Y-K) < 6\,v_8 \cdot t_F(\pi(L)),\]
if $F$ has genus bigger than one. Thus in both cases the result follows immediately by Lemma ~\ref{lem:twistNumber}.
\end{proof}
\begin{rem}{\rm 
Theorem \ref{Thm:FxIUpperBound} is the analogue of the ``volumish theorem" of \cite{VTJPAK} for alternating links in thickened surfaces, where the authors relay on the two sided bounds volume bounds in terms
of the twist number of alternating projections given  by Lackenby  \cite{VHLC}.}
\end{rem}

\begin{rem} {\rm  In \cite{CK}, Champanerkar and Kofman show that if $\pi(L)$ is an alternating projection as in 
Theorem ~\ref{Thm:FxIUpperBound}, then $\pi(L)$ admits two sided linear bounds in terms of coefficients of a specialization
of the Krushkal polynomial \cite{Slava}. Then they also combine this with equation (\ref{eq:tn}) to conclude that 
Krushkal's polynomial also gives two-sided bounds of the volume of alternating links in thickened surfaces.  Krushkal informed the authors
that Andrew Will \cite{WA} also obtained a similar result.  Their approach, however, doesn't lead to a proof of invariance of $t_F(\pi(L))$.}
\end{rem}

\begin{rem} {\rm  In \cite{CVWGA} Bavier shows that if  $M$ is closed  and $\pi(K)$  is a weakly generalized  alternating knot projection, that is twist-reduced,  on a surface $F\subset M$ of genus at least 1,  then
 the twist number $t_F(\pi(K)$ provides 2-sided bounds on the cusp volume of $M\setminus K.$ We close the section by noting that,  by Lemma ~\ref{lem:twistNumber}, Theorem 1.1 of  \cite{CVWGA}
 and the resulting applications to Dehn filling given therein can also be stated in terms of the  skein theoretic quantity $J_0(\pi(K))$.}
 \end{rem}

\bibliographystyle{hamsplain}
\bibliography{WAGFundamentarev}

\end{document}